\documentclass{amsart}
\usepackage{amsfonts}

\newtheorem{theorem}{Theorem}
\theoremstyle{plain}
\newtheorem{acknowledgement}{Acknowledgement}

\newtheorem{corollary}{Corollary}

\newtheorem{lemma}{Lemma}

\newtheorem{proposition}{Proposition}
\newtheorem{remark}{Remark}

\numberwithin{equation}{section}
\input{tcilatex}

\begin{document}
\title[multivariate q-Normal]{On three dimensional multivariate version of
q-Normal distribution and probabilistic interpretations of Askey--Wilson,
Al-Salam--Chihara and q-ultraspherical polynomials .}
\author{Pawe\l\ J. Szab\l owski }
\address{Emeritus in Department of Mathematics and Information Sciences,
Warsaw University of Technology ul Koszykowa 75, 00-662 Warsaw, Poland }
\date{September 2017 }

\begin{abstract}
We study properties of compactly supported, 4 parameter \newline
$(\rho _{12},\rho _{23},\rho _{13},q)\in (-1,1)^{\times 4}$ family of
continuous type 3 dimensional distributions, that have the property that for 
$q\rightarrow 1^{-}$ this family tends to some 3 dimensional Normal
distribution. For $q=0$ we deal with 3 dimensional generalization of
Kesten--McKay distribution. In a very special case when $\rho _{12}\rho
_{13}\rho _{23}=q$ all one dimensional marginals are identical, semicircle
distributions. We find both all marginal as well as all conditional
distributions. Moreover, we find also families of polynomials that are
orthogonalized by these one-dimensional margins and one-dimensional
conditional distributions. Consequently, we find moments of both conditional
and unconditional distributions of dimensions one and two. In particular, we
show that all one-dimensional and two-dimensional conditional moments of,
say, order $n$ and are polynomials of the same order $n$ in the conditioning
random variables. Finding above mentioned orthogonal polynomials leads us to
a probabilistic interpretation of these polynomials. Among them are the
famous Askey-Wilson, Al-Salam--Chihara polynomials considered in the
complex, but conjugate, parameters, as well as q-Hermite and Rogers
polynomials. It seems that this paper is one of the first papers that give a
probabilistic interpretation of Rogers (continuous q-ultraspherical)
polynomials.
\end{abstract}

\maketitle

\section{Introduction}

The purpose of this note is double. Firstly, to present a further step in
the search for the non trivial generalizations of the Normal (Gaussian)
density. The unique properties of the Normal distribution are well known and
found to be useful in many, not only stochastic, applications. Let us
mention only one such property that is in fact well known, however, is less
popularized and which is one of the most useful. Namely, all conditional
mixed moments are polynomials of the same order in the conditioning random
variables. To formulate it precisely, let us introduce the following
notation. Let $\mathcal{F}$ denote some set of indices, then $\sigma
(X_{j},j\notin \mathcal{F})$ we mean the $\sigma -$field generated by all
random variables $X_{j}$ whose indices do not belong the set $\mathcal{F}$.
Now the property in question can be expressed in the precise fashion. Let us
assume that a vector $(X_{1},...,X_{n})$ has joint Normal distribution. Let $%
\mathcal{F}$ be a subset of $\left\{ 1,...,n\right\} ,$ then the conditional
expectation given by the expression $E(\prod_{i\in \mathcal{F}%
}X_{i}^{m_{i}}|\sigma (X_{j},j\notin \mathcal{F}))$ is a polynomial of order 
$\sum_{i\in \mathcal{F}}m_{i}$ in $X_{j},$ $j\notin \mathcal{F}$. The above
mentioned property happens to all nontrivial subsets $\mathcal{F}$ that is
why we will say that Normal distribution has polynomial conditional moments
property (PCM).

This beautiful property has numerous applications is the theory of Gaussian
stochastic processes. In particular, is strongly reflected in the properties
of Wiener and related processes whose role in the contemporary analysis is
difficult to overestimate.

The second intention of this note is to give a probabilistic interpretation
of certain, well known in other areas of mathematical analysis, families of
orthogonal polynomials. More precisely, we mean certain families of
polynomials that appear in the so called $q-$series theory and are related
to the so called Askey-Wilson polynomials. There were many interpretations
of these polynomials in combinatorics (to mention only works of Sylvie
Corteel). It was a great surprise to R. Askey to discover that the so called 
$q-$Hermite polynomials have nice probabilistic interpretation shown in the
work of W. Bryc \cite{Bryc2001S}. To the astonishment of R. Askey there was
no $q$ in the original formulation of the probabilistic problem. Later using
the same model Bryc, Matysiak and Szab\l owski were able to interpret
probabilistically the so called Al-Salam--Chihara polynomials. For details,
see \cite{bms}.

We are not that smart as W. Bryc and are not able to eliminate a variable $q$
from the formulation of the problem. However, we are able to define a
three-dimensional, compactly supported, continuous type distribution,
calculate its all marginal unconditional as well as conditional
distributions. Moreover, we are able to find conditional moments and show
that the analyzed distribution has PCM property.

We are also able to find polynomials that are orthogonalized by the
one-dimensional marginals (Rogers polynomials) and by the one-dimensional
conditional $X|(Y=y,Z=z)$ (Askey--Wilson polynomials).

We notice also that as the two-dimensional marginals appear products of the
so called $q-$conditional normal ($q-$CN) distribution analyzed say in \cite%
{Szab5} and, as mentioned above, as one-dimensional marginal Rogers
distribution. As a special case we get either the so called generalized
Kesten--McKay (recently analyzed in \cite{Szab18}) or the semicircle
distributions. Both of them have numerous applications. Especially the
semicircle distribution and related distributions lie in the center of the
so called free probability a branch of non-commutative probability recently
dynamically developing.

The paper is organized as follows. In the next section we introduce the
analyzed density. Next we introduce notation that describes it briefly. Then
in subsection \ref{aux} we provide definitions of the families of
polynomials that will appear in the sequel and also some auxiliary facts
concerning them.

Section, \ref{main} presents our main results. It consists of two
subsections the first one presenting marginal distributions and the families
of polynomials that are orthogonal with respect to them and the second
presenting conditional distributions and families of polynomials orthogonal
with respect to them. In particular, we are able to prove there PCM property
of the analyzed distribution. The last Section, \ref{open} presents some
remarks and poses some open problems.

\section{Multivariate density}

We will study the following $3-$dimensional density: 
\begin{gather*}
f_{3D}(x,y,z|\rho _{12},\rho _{13},\rho _{23},q)=f_{N}\left( x|q\right)
f_{N}(y|q)f_{N}(z|q) \\
\times \frac{C_{3D}\left( \rho _{12}^{2}\right) _{\infty }\left( \rho
_{13}^{2}\right) _{\infty }\left( \rho _{23}^{2}\right) _{\infty }}{%
\prod_{i=0}^{\infty }(\omega _{q}\left( x,y|\rho _{12}q^{i}\right) \omega
_{q}\left( x,z|\rho _{13}q^{i}\right) \omega _{q}\left( y,z|\rho
_{23}q^{i}\right) )}\text{,}
\end{gather*}%
where 
\begin{equation*}
f_{N}(x|q)=\frac{(q)_{\infty }\sqrt{1-q}\sqrt{4-(1-q)x^{2}}}{2\pi }%
\prod_{i=1}^{\infty }l_{q}\left( x|q^{i}\right) \text{,}
\end{equation*}%
with 
\begin{gather*}
l_{q}\left( x|a\right) =(1+a)^{2}-(1-q)ax^{2}\text{, } \\
\omega _{q}\left( x,y|\rho \right) =(1-\rho ^{2})^{2}-(1-q)xy\rho (1+\rho
^{2})+(1-q)\rho ^{2}\left( x^{2}+y^{2}\right) \text{,}
\end{gather*}%
$\left\vert q\right\vert ,\left\vert \rho _{12}\right\vert ,\left\vert \rho
_{13}\right\vert ,\left\vert \rho _{23}\right\vert <1$, $\left\vert
x\right\vert ,\left\vert y\right\vert ,\left\vert z\right\vert \leq 2/\sqrt{%
1-q}\text{,}$ and $C_{3D}$ is a suitable constant. Symbol $(a)_{n}$ will be
explained below.

Properties, alternative forms and ways of simulation of $f_{N}$ were
presented in \cite{Szab5}.

Ismail et al. in \cite{ISV87} proved rigorously that $f_{N}(x|q)\allowbreak
\rightarrow \allowbreak \frac{1}{\sqrt{2\pi }}\exp (-x^{2}/2)$ as $%
q\rightarrow 1^{-}.$

\subsection{Notation}

We will use the following denotations%
\begin{eqnarray*}
\lbrack 0]_{q} &=&0,~~[n]_{q}=1+...+q^{n-1},~~n\geq 1, \\
\lbrack n]_{q}! &=&\left\{ 
\begin{array}{ccc}
1 & if & n=0 \\ 
\prod_{i=1}^{n}[i]_{q} & if & n\geq 1%
\end{array}%
\right. ,\text{ }S\left( q\right) =\left\{ 
\begin{array}{ccc}
\lbrack -\frac{2}{\sqrt{1-q}},\frac{2}{\sqrt{1-q}}] & if & \left\vert
q\right\vert <1 \\ 
\mathbb{R} & if & q=1%
\end{array}%
\right. , \\
\QATOPD[ ] {n}{k}_{q} &=&\left\{ 
\begin{array}{ccc}
\frac{\lbrack n]_{q}!}{[k]_{q}![n-k]_{q}!} & if & 0\leq k\leq n \\ 
0 & if & \text{otherwise}%
\end{array}%
\right. 
\end{eqnarray*}%
and the so called $q-$Pochhammer symbol: 
\begin{equation*}
(a)_{j}=\left\{ 
\begin{array}{ccc}
\prod_{k=1}^{j}(1-aq^{k-1}) & if & j\geq 1 \\ 
0 & if & j=0%
\end{array}%
\right. \text{.}
\end{equation*}%
Let us note that $f_{3D}$ is nonnegative on the cube $S(q)^{\times 3}$.
Hence it can be viewed as a density.

It is obvious to notice that 
\begin{equation*}
(q)_{n}=(1-q)^{n}[n]_{q}!\text{.}
\end{equation*}

For $q=0$ we have 
\begin{equation*}
f_{3D}(x,y,z|\rho _{12},\rho _{13},\rho _{23},0)=C_{3D}\frac{\sqrt{%
(4-x^{2})(4-y^{2})(4-z^{2})}}{\omega _{0}(x,y|\rho _{12})\omega
_{0}(x,z|\rho _{13})\omega _{0}(y,z|\rho _{23})}\text{,}
\end{equation*}%
which is a variation of the multivariate generalization of Kesten--MacKay
distribution (compare \cite{Szab18} formula 2.8 ), while by taking $%
q\rightarrow 1^{-}$ we arrive at $3$ dimensional Normal distribution with
zero expectations and the following variance-covariance matrix :%
\begin{equation*}
\left[ 
\begin{array}{ccc}
\frac{1+r}{1-r} & \frac{\rho _{12}+\rho _{13}\rho _{23}}{1-r} & \frac{\rho
_{13}+\rho _{12}\rho _{23}}{1-r} \\ 
\frac{\rho _{12}+\rho _{13}\rho _{23}}{1-r} & \frac{1+r}{1-r} & \frac{\rho
_{23}+\rho _{12}\rho _{13}}{1-r} \\ 
\frac{\rho _{13}+\rho _{12}\rho _{23}}{1-r} & \frac{\rho _{23}+\rho
_{12}\rho _{13}}{1-r} & \frac{1+r}{1-r}%
\end{array}%
\right] \text{,}
\end{equation*}%
where $r=\rho _{12}\rho _{13}\rho _{23}\text{. This fact follows (\ref{eiv}%
), (\ref{cov}) and the properties of the function }f_{N}(x|q).$

\subsection{Some auxiliary facts\label{aux}}

The following facts were selected from \cite{IA}, \cite{Szab-rev} and \cite%
{Szab-bAW}.

We will use the following families of orthogonal polynomials:

\subsubsection{$q-$Hermite polynomials $\left\{ H_{j}(x|q)\right\} _{j\geq
-1}$.}

They satisfy the following recursive equations: \newline

\begin{equation*}
H_{n+1}(x|q)=xH_{n}(x|q)-[n]_{q}H_{n-1}(x|q)\text{,}
\end{equation*}%
with $H_{-1}(x|q)=0\text{,}$ $H_{0}(x|q)=1.$ They are orthogonal with
respect to the density $f_{N}(x|q)$ i.e. we have: 
\begin{equation}
\int_{S(q)}H_{i}(x|q)H_{j}(x|q)f_{N}(x,q)dx=\delta _{ij}[i]_{q}!\text{,}
\label{oH}
\end{equation}%
where $\delta _{ij}$ is the Kronecker's delta. In the sequel the following
Poisson-Mehler summation formula will be of great help: 
\begin{equation}
\frac{\left( \rho ^{2}\right) _{\infty }}{\prod_{i=0}^{\infty }\omega
_{q}\left( x,y|\rho q^{i}\right) }=\sum_{j\geq 0}\frac{\rho ^{j}}{_{[j]_{q}!}%
}H_{j}(x|q)H_{j}(y|q),  \label{P-M}
\end{equation}%
valid for all $\left\vert \rho \right\vert <1\text{,}$ $x,y\in S(q)\text{. }$%
\newline

Notice that we have useful formula:%
\begin{equation}
\sum_{j\geq 0}\frac{(\rho q)^{j}}{_{[j]_{q}!}}H_{j}(x|q)H_{j}(y|q)%
\allowbreak =\allowbreak \frac{\omega _{q}(x,y|\rho )}{(1-\rho ^{2})(1-\rho
^{2}q)}\sum_{j\geq 0}\frac{\rho ^{j}}{_{[j]_{q}!}}H_{j}(x|q)H_{j}(y|q),
\label{P-Mq}
\end{equation}%
for all $\left\vert \rho \right\vert <1\text{,}$ $x,y\in S(q),$ since $\frac{%
\left( q^{2}\rho ^{2}\right) _{\infty }}{\prod_{i=0}^{\infty }\omega
_{q}\left( x,y|\rho q^{i+1}\right) }\allowbreak =\allowbreak \frac{\omega
_{q}(x,y|\rho )}{(1-\rho ^{2})(1-\rho ^{2}q)}\frac{\left( \rho ^{2}\right)
_{\infty }}{\prod_{i=0}^{\infty }\omega _{q}\left( x,y|\rho q^{i}\right) }.$

$\text{Generalizations of (\ref{P-M}) and their consequences are presented
in \cite{SzabP-M}.}$

The following formula will also be of help:

\begin{lemma}
\label{3her}We have 
\begin{gather}
\int_{S(q)}H_{k}(x|q)H_{m}(x|q)H_{n}(x|q)f_{N}(x|q)dx=  \label{3pol} \\
\left\{ 
\begin{array}{ccc}
0 & if & 
\begin{array}{c}
k+m+n\;\text{is odd or}\; \\ 
k+m<n\;\text{or}\;k+n<m \\ 
\;\text{or}\;n+m<k%
\end{array}
\\ 
\frac{\lbrack m]_{q}![n]_{q}![k]_{q}!}{[\frac{m+n-k}{2}]_{q}![\frac{m+k-n}{2}%
]_{q}![\frac{n+k-m}{2}]_{q}!} & if & \;\text{otherwise}\;%
\end{array}%
\right. .  \notag
\end{gather}
\end{lemma}

\begin{proof}
We use linearization formula 3.13 of \cite{Szab-rev}. It is obvious that the
value of the integral is nonzero when $n+m-2j=k$ for some $j\geq 0\text{,}$
which means $m+n+k$ must be even and $n+m\geq k\text{.}$ Note that formula
must be symmetric with respect to $m,n,k\text{.}$ Then this value is equal
to $\QATOPD[ ] {n}{\frac{n+m-k}{2}}_{q}\QATOPD[ ] {m}{\frac{n+m-k}{2}}_{q}[%
\frac{n+m-k}{2}]_{q}![k]_{q}!$ which is equal to (\ref{3pol}).
\end{proof}

\subsubsection{Al-Salam--Chihara polynomials $\left\{ P_{n}(x|y,\protect\rho %
,q)\right\} _{n\geq -1}$.}

They satisfy the following recursive equations: 
\begin{equation}
P_{n+1}(x|y,\rho ,q)=(x-\rho yq^{n})P_{n}(x|y,\rho ,q)-(1-\rho
^{2}q^{n-1})[n]_{q}P_{n-1}(x|y,\rho ,q),  \label{asc}
\end{equation}%
with $P_{-1}(x|y,\rho ,q)=0\text{,}$ $P_{0}(x|y,\rho ,q)=1.$ They turn out
to be orthogonal with respect to the following densities: 
\begin{equation}
f_{CN}(x|y,\rho ,q)=f_{N}(x|q)\frac{(\rho ^{2})_{\infty }}{%
\prod_{i=0}^{\infty }\omega _{q}(x,y|\rho q^{i})}\text{.}  \label{CN}
\end{equation}%
That is we have: 
\begin{equation*}
\int_{S\left( q\right) }P_{n}\left( x|y,\rho ,q\right) P_{m}\left( x|y,\rho
,q\right) f_{CN}\left( x|y,\rho ,q\right) dx=\left\{ 
\begin{array}{ccc}
0 & when & n\neq m \\ 
\left( \rho ^{2}\right) _{n}\left[ n\right] _{q}! & when & n=m%
\end{array}%
\right. \text{.}
\end{equation*}%
Moreover, we know that the densities $f_{CN}$ have the following interesting
property (Chapman--Kolmogorov property): 
\begin{equation}
\int_{S\left( q\right) }f_{CN}\left( x|y,\rho _{1},q\right) f_{CN}\left(
y|z,\rho _{2},q\right) dy=f_{CN}\left( x|z,\rho _{1}\rho _{2},q\right) .
\label{C-K}
\end{equation}%
Probabilistic aspects of Al-Salam--Chihara polynomials are presented in \cite%
{bms}.

Notice that for $q\allowbreak \rightarrow \allowbreak 1^{-}$ we have $%
\forall n\geq 0\allowbreak :\allowbreak $%
\begin{equation*}
P_{n}(x|y,\rho ,q)\allowbreak \rightarrow \allowbreak H_{n}(\frac{x-\rho y}{%
\sqrt{1-\rho ^{2}}})(1-\rho ^{2})^{n/2}
\end{equation*}%
and $H_{n}(x|q)\allowbreak \rightarrow \allowbreak H_{n}(x)$, where $H_{n}$
denotes the so-called probabilistic Hermite polynomials i.e. monic,
orthogonal with respect to $\exp (-x^{2}/2).$ Consequently we deduce that 
\begin{equation*}
f_{CN}(x|y,\rho ,q)\allowbreak \rightarrow \allowbreak \exp (-\frac{(x-\rho
y)^{2}}{2(1-\rho ^{2})})/\sqrt{2\pi (1-\rho ^{2}},
\end{equation*}%
as $q\rightarrow 1^{-}$ by (\ref{CN}).

\begin{remark}
Observe that making use of the density $f_{CN}$ as well as (\ref{P-M}) we
deduce that $f_{3D}$ can be presented in one of the following two equivalent
forms: 
\begin{equation}
f_{3D}(x,y,z|\rho _{12},\rho _{13},\rho _{23},q)=C_{3D}f_{CN}(x|y,\rho
_{12},q)f_{CN}(y|z,\rho _{23},q)f_{CN}(z|x,\rho _{13},q).  \label{prod_c}
\end{equation}%
and also that: 
\begin{gather}
f_{3D}(x,y,z|\rho _{12},\rho _{13},\rho
_{23},q)=C_{3D}f_{N}(x|q)f_{N}(y|q)f_{N}(z|q)  \label{3PM} \\
\times \sum_{j,k,l\geq 0}\frac{\rho _{12}^{j}\rho _{23}^{k}\rho _{13}^{l}}{%
[j]_{q}![k]_{q}![l]_{q}!}%
H_{j}(x|q)H_{j}(y|q)H_{k}(y|q)H_{k}(z|q)H_{l}(x|q)H_{l}(z|q).  \notag
\end{gather}

Hence we see that for $q\allowbreak \rightarrow \allowbreak 1^{-}$ the
density $f_{3D}(x,y,z|\rho _{12},\rho _{13},\rho _{23},q)$ tends to the
density of Normal distribution with the variance-covariance matrix given in
the introduction above.
\end{remark}

\subsubsection{Continuous $q-$ultraspherical (Rogers) polynomials}

Name Rogers polynomials will be used for brevity. By Rogers polynomials it
is meant (see e.g. \cite{IA}) family of polynomials $\left\{ C_{n}(x|\beta
,q)\right\} _{n\geq -1}$ defined by the following three term recurrence : 
\begin{equation}
2x(1-\beta q^{n})C_{n}(x|\beta ,q)=(1-q^{n+1})C_{n+1}(x|\beta ,q)+(1-\beta
^{2}q^{n-1})C_{n-1}(x|\beta ,q)\text{,}  \label{3tc}
\end{equation}%
with $C_{-1}(x|\beta ,q)=0\text{,}$ $C_{0}(x|\beta ,q)=1$, $\left\vert \beta
\right\vert <1$. In fact we will need these polynomials modified in the
following way: 
\begin{equation}
C_{n}(x|\beta ,q)=\frac{(1-q)^{n/2}(\beta )_{n}}{(q)_{n}}R_{n}(\frac{2x}{%
\sqrt{1-q}}|\beta ,q)\text{.}  \label{newC}
\end{equation}%
By inserting (\ref{newC}) into (\ref{3tc}) and denoting $y=\frac{2x}{\sqrt{%
1-q}}\text{,}$ and canceling out $(\beta )_{n+1}(1-q)^{(n+1)/2}/(q)_{n},$
finally after little algebra, we end up with the following three term
recurrence: 
\begin{equation}
yR_{n}(y|\beta ,q)=R_{n+1}(y|\beta ,q)+[n]_{q}\frac{(1-\beta ^{2}q^{n-1})}{%
(1-\beta q^{n-1})(1-\beta q^{n})}R_{n-1}(y|\beta ,q),  \label{Rp}
\end{equation}%
since $\left( \beta \right) _{n+1}=\left( \beta \right) _{n}(1-\beta q^{n})%
\text{,}$ $(q)_{n}=(q)_{n-1}(1-q^{n})\text{,}$ $[n]_{q}=(1-q^{n})/(1-q)$. It
is known also (see e.g. \cite{Szab-rev}(2.37) or in non modified form \cite%
{IA}(13.2.4)) that the measure that makes these polynomials orthogonal has
the following density: 
\begin{equation}
f_{R}(x|\beta ,q)=f_{N}(x|q)\frac{(\beta ^{2})_{\infty }}{(\beta )_{\infty
}(\beta q)_{\infty }\prod_{i=0}^{\infty }l_{q}(x|\beta q^{i})}\text{.}
\label{fR}
\end{equation}%
To simplify notation we will call distribution with the density $f_{R}$ the
Rogers distribution. Notice that it is symmetric.

Note also that when $\beta \allowbreak =\allowbreak q$ then the equation (%
\ref{Rp}) is simplified to 
\begin{equation*}
yR_{n}(y|q,q)\allowbreak =\allowbreak R_{n+1}(y|q,q)+\frac{1}{1-q}%
R_{n-1}(y|q,q).
\end{equation*}%
Hence $R_{n}(y|q,q)\allowbreak =\allowbreak U_{n}(y\sqrt{1-q}/2)/(1-q)^{n/2}$%
, where $U_{n}(x)$ denotes the Chebyshev polynomial of the second kind,
since polynomials $\left\{ U_{n}\right\} $ satisfy the following three term
recurrence:%
\begin{equation*}
2xU_{n}(x)\allowbreak =\allowbreak U_{n+1}(x)+U_{n-1}(x),
\end{equation*}%
with $U_{-1}(x)\allowbreak =\allowbreak 0$ and $U_{0}(x)\allowbreak
=\allowbreak 1.$ Consequently one can see that 
\begin{equation*}
f_{R}(x|q,q)\allowbreak =\allowbreak \frac{\sqrt{1-q}\sqrt{4-(1-q)x^{2}}}{%
2\pi },
\end{equation*}%
$x\in S(q).$

It is worth to mention the following formula : 
\begin{equation}
\sum_{i\geq 0}\frac{r^{i}}{[i]_{q}!}H_{i+k}(x|q)H_{i+m}(x|q)=W_{k,m}(x|r,q)%
\times \sum_{i\geq 0}\frac{r^{i}}{[i]_{q}!}H_{i}(x|q)H_{i}(x|q).
\label{simp}
\end{equation}

where $W_{k,m}(x|r,q)$ is given by the following formula: 
\begin{equation}
W_{k,m}(x|r,q)=\sum_{s=0}^{k}\frac{q^{\binom{s}{2}}(-r)^{s}(r)_{m+s}}{%
(r^{2})_{m+s}}H_{k-s}(x|q)R_{m+s}(x|r,q).  \label{ww}
\end{equation}%
From its definition it follows that $W_{k,m}(x|r,q)=W_{m,k}(x|r,q)\text{,}$ $%
k,m\geq 0.$

In particular taking into account (\ref{P-Mq}) we have useful summation
formula: 
\begin{eqnarray*}
&&\sum_{i\geq 0}\frac{(qr)^{i}}{[i]_{q}!}H_{i+k}(x|q)H_{i+m}(x|q) \\
&=&W_{k,m}(x|rq,q)\frac{\omega _{q}(x,x|r)}{(1-r^{2})(1-qr^{2})}\times
\sum_{i\geq 0}\frac{r^{i}}{[i]_{q}!}H_{i}(x|q)H_{i}(x|q).
\end{eqnarray*}

(\ref{simp}) follows a formula in \cite{SzablAW}(Lemma 3 i)), where we set $%
x=y$ and the following observation concerning relationship between
polynomials $\left\{ P_{n}(x|x,\rho ,q)\right\} _{i\geq -1}$ and $\left\{
R_{n}(x|\rho ,q)\right\} _{i\geq -1}$ that: 
\begin{equation*}
R_{n}(x|\rho ,q)=P_{n}\left( x|x,\rho ,q\right) /(\rho )_{n}\text{.}
\end{equation*}

\section{Main results\label{main}}

\subsection{Marginals and moments}

\subsubsection{\textbf{Marginal distributions}}

We will find $C_{3D}$ as well as marginal densities first. We have the
following result:

\begin{theorem}
\label{margin}Let us denote for simplicity $r=\rho _{12}\rho _{13}\rho _{23}%
\text{.}$ Then

i) $C_{3D}=1-r$.

ii) two- dimensional marginals depend on two parameters (except for $q)$ in
fact. In the case of $f_{YZ}$ on $\rho _{23}$ and $\rho _{12}\rho _{13}$
only. 
\begin{gather*}
f_{YZ}(y,z|\rho _{12},\rho _{13},\rho _{23},q)=\int_{S(q)}f_{3D}(x,y,z|\rho
_{12},\rho _{13},\rho _{23},q)dx= \\
(1-r)f_{N}(y|q)f_{N}(z|q)\frac{\left( \rho _{23}^{2}\right) _{\infty }(\rho
_{12}^{2}\rho _{13}^{2})_{\infty }}{\prod_{i=0}^{\infty }\omega \left(
y,z|\rho _{23}q^{i}\right) \omega \left( y,z|\rho _{12}\rho
_{13}q^{i}\right) } \\
=(1-r)f_{CN}(y|z,\rho _{23},q)f_{CN}(z|y,\rho _{12}\rho _{13},q)
\end{gather*}%
and similarly for $f_{XZ}\text{,}$ and $f_{XY}\text{.}$

iii) Marginal one dimensional densities $\int_{S(q)}\int_{S(q)}f_{3D}(x,y,z|%
\rho _{12},\rho _{13},\rho _{23},q)dxdy=f_{Z}(z|\rho _{12},\rho _{13},\rho
_{23},q)$ depend on the product $r=\rho _{12}\rho _{23}\rho _{13}$ only.
More over we have $f_{Z}(z|\rho _{12},\rho _{13},\rho _{23},q)\allowbreak
=\allowbreak f_{R}(z|r,q),$ where $f_{R}$ is a Rogers distribution given by (%
\ref{fR}).
\end{theorem}

\begin{proof}
We will use representation (\ref{prod_c}) and the property (\ref{C-K}).
First, let us calculate the integral: 
\begin{gather*}
f_{YZ}(y,z|\rho _{12},\rho _{13},\rho _{23},q)=\int_{S(q)}f_{3D}(x,y,z|\rho
_{12},\rho _{13},\rho _{23},q)dx= \\
C_{3D}f_{CN}(y|z,\rho _{23},q)\int_{S(q)}f_{CN}(x|y,\rho
_{12},q)f_{CN}(z|x,\rho _{13},q)dx \\
=C_{3D}f_{CN}(y|z,\rho _{23},q)f_{CD}(z|y,\rho _{12}\rho _{13},q) \\
=C_{3D}f_{N}(y|q)f_{N}(z|q)\frac{\left( \rho _{23}^{2}\right) _{\infty
}(\rho _{12}^{2}\rho _{13}^{2})_{\infty }}{\prod_{i=0}^{\infty }\omega
_{q}\left( y,z|\rho _{23}q^{i}\right) \omega _{q}\left( y,z|\rho _{12}\rho
_{13}q^{i}\right) }\text{.}
\end{gather*}%
Now let us calculate integral 
\begin{gather*}
f_{Z}(z|\rho _{12},\rho _{13},\rho
_{23},q)=\int_{S(q)}\int_{S(q)}f_{3D}(x,y,z|\rho _{12},\rho _{13},\rho
_{23},q)dxdy= \\
C_{3D}f_{N}(z|q)\int_{S(q)}\frac{\left( \rho _{23}^{2}\right) _{\infty
}(\rho _{12}^{2}\rho _{13}^{2})_{\infty }}{\prod_{i=0}^{\infty }\omega
_{q}\left( y,z|\rho _{23}q^{i}\right) \omega _{q}\left( y,z|\rho _{12}\rho
_{33}q^{i}\right) }f_{N}(y|q)dy \\
=C_{3D}f_{N}(z|q)\int_{S(q)}f_{N}(y|q)(\sum_{k,j\geq 0}\frac{(\rho _{12}\rho
_{13})^{j}\rho _{23}^{k}}{[j]_{q}![k]_{q}!}%
H_{j}(z|q)H_{j}(y|q)H_{k}(y|q)H_{k}(z|q))dy \\
=C_{3D}f_{N}(z|q)\sum_{j\geq 0}\frac{(\rho _{12}\rho _{13}\rho _{23})^{j}}{%
[j]_{q}!}H_{j}^{2}(z|q)\text{.}
\end{gather*}%
Now notice that using (\ref{P-M}) and noticing that $\omega _{q}\left(
x,x|r\right) =(1-r)^{2}l_{q}(x|r)$ 
\begin{eqnarray*}
\sum_{j\geq 0}\frac{(r)^{j}}{[j]_{q}!}H_{j}^{2}(z|q) &=&\frac{%
(r^{2})_{\infty }}{\prod_{i=0}^{\infty }(\omega _{q}\left( z,z|rq^{i}\right) 
}=\frac{(r^{2})_{\infty }}{\prod_{i=0}^{\infty }(1-rq^{i})^{2}l_{q}(z|rq^{i})%
} \\
&&\frac{(r^{2})_{\infty }}{(r)_{\infty }^{2}\prod_{i=0}^{\infty }l(z|rq^{i})}%
\text{.}
\end{eqnarray*}%
Comparing this with (\ref{fR}) we see that 
\begin{equation*}
f_{Z}(z|\rho _{12},\rho _{13},\rho _{23},q)=C_{3D}f_{R}(z|r,q)/(1-r),
\end{equation*}%
since $(r)_{\infty }\allowbreak =\allowbreak (1-r)(rq)_{\infty }$. Now since 
$f_{R}$ is the density we deduce that $C_{3D}=1-\rho _{12}\rho _{13}\rho
_{23}\allowbreak =\allowbreak 1-r$.
\end{proof}

\begin{proposition}
The densities $f_{Z}(z|\rho _{12},\rho _{13},\rho _{23},q)$, $f_{Y}$ and $%
f_{X}$ are of one of the following equivalent forms: 
\begin{gather}
(1-r)f_{N}(z|q)\sum_{j\geq 0}\frac{r^{j}}{[j]_{q}!}%
H_{j}^{2}(z|q)=(1-r)f_{CN}(z|z,r,q)  \label{1wym} \\
=(1-r)f_{N}(z|q)\frac{(r^{2})_{\infty }}{(r)_{\infty
}^{2}\prod_{i=0}^{\infty }l_{q}(z|rq^{i})}  \label{roge} \\
=(1-r)f_{N}(z|q)\sum_{k=0}^{\infty }\frac{r^{k}}{[k]_{q}!(r)_{k+1}}%
H_{2k}(x|q)  \label{r2} \\
=\frac{2(1+r)\sqrt{4-(1-q)z^{2}}\left( q\right) _{\infty }(r^{2}q)_{\infty }%
}{\pi l_{q}(z|r)(rq)_{\infty }^{2}}\prod_{j=1}^{\infty }\frac{l_{q}(z|q^{j})%
}{l_{q}(z|rq^{j})}.  \label{r3}
\end{gather}
\end{proposition}

\begin{proof}
Notice that we have: 
\begin{equation*}
\sum_{j\geq 0}\frac{r^{j}}{[j]_{q}!}H_{j}^{2}(z|q)=\sum_{j\geq 0}\frac{r^{j}%
}{[j]_{q}!}\sum_{k=0}^{j}\frac{([j]_{q}!)^{2}}{([k]_{q}!)^{2}[j-k]_{q}!}%
H_{2k}(x|q)\text{,}
\end{equation*}%
since $H_{j}^{2}(x|q)=\sum_{k=0}^{j}\frac{([j]_{q}!)^{2}}{%
([k]_{q}!)^{2}[j-k]_{q}!}H_{2k}(x|q)$ by \cite{Szab-rev} (3.13). Hence 
\begin{gather*}
f_{Z}(z|\rho _{12},\rho _{13},\rho
_{23},q)=(1-r)f_{N}(z|q)\sum_{k=0}^{\infty }\frac{r^{k}}{[k]_{q}!}%
H_{2k}(x|q)\sum_{j\geq k}r^{j-k}\QATOPD[ ] {j}{k}_{q} \\
=\allowbreak (1-r)f_{N}(z|q)\sum_{k=0}^{\infty }\frac{r^{k}}{%
[k]_{q}!(r)_{k+1}}H_{2k}(x|q),
\end{gather*}%
because $\sum_{m\geq 0}r^{m}\QATOPD[ ] {m+k}{k}_{q}=1/(r)_{k+1}$.

To get third, multiplicative form we argue as follows 
\begin{gather*}
f_{Z}(z|r,q)=(1-r)f_{N}(z|q)\sum_{j\geq 0}\frac{r^{j}}{[j]_{q}!}%
H_{j}^{2}(z|q) \\
=\frac{2(1-r^{2})\sqrt{4-(1-q)z^{2}}\left( q\right) _{\infty
}(r^{2})_{\infty }}{\pi l_{q}(z|r)(r)_{\infty }^{2}}\prod_{j=1}^{\infty }%
\frac{l_{q}(z|q^{j})}{l_{q}(z|rq^{j})} \\
=\frac{2\sqrt{4-(1-q)z^{2}}(1+r)(q)_{\infty }(r^{2}q)_{\infty }}{\pi
l_{q}(z|r)(rq)_{\infty }^{2}}\prod_{j=1}^{\infty }\frac{l_{q}(z|q^{j})}{%
l_{q}(z|rq^{j})}\text{.}
\end{gather*}
\end{proof}

\begin{proposition}
\label{ort} Polynomials that are orthogonal with respect to $f_{R}$ belong
to the family $\left\{ w_{n}(x|r,q)\right\} _{n\geq -1}$ of monic Rogers
polynomials satisfying three term recurrence given by (\ref{Rp}) with $\beta
=r$. Notice that $w_{n}(x|r,q)=P_{n}(x|x,r,q)/(r)_{n}$. Moreover we have 
\begin{equation*}
\int_{S(q)}w_{n}(x|r,q)w_{m}(x|r,q)f_{R}(x|r,q)=\left\{ 
\begin{array}{ccc}
0 & if & n\neq m \\[0pt] 
\lbrack n]_{q}!\frac{(1-r)(r^{2})_{n}}{(r)_{n}(r)_{n+1}} & if & m=n%
\end{array}%
\right. \text{.}
\end{equation*}
\end{proposition}

\begin{proof}
Consider (\ref{Rp}) and apply formula (2.2.18) of \cite{IA}.
\end{proof}

\subsubsection{\textbf{Moments}}

Now let us calculate some moments of these variables. Notice that since
function $l$ depends on $z^{2}$ we deduce that $f_{Z}(z|r,q)$ is symmetric
in $z$, consequently all odd moments of $Z$ are equal to zero.

Further, we have the following lemma:

\begin{lemma}
\label{even}i) Suppose $Z\sim f_{Z}(z|r,q)\text{,}$ with $\left\vert
r\right\vert ,\left\vert q\right\vert <1\text{,}$ $z\in S(q)$, then 
\begin{equation*}
EH_{2n}(Z|q)=\frac{r^{n}[2n]_{q}!}{[n]_{q}!(rq)_{n}}\text{.}
\end{equation*}%
In particular 
\begin{equation}
EZ=0,\;\text{{}}\;\limfunc{var}(Z)=EZ^{2}=\frac{1+r}{1-rq}\text{.}
\label{eiv}
\end{equation}

ii) Let $(Y,Z)\sim f_{YZ}(y,z|\rho _{12},\rho _{13},\rho _{23},q)$, then 
\begin{gather*}
EH_{m}(Y|q)H_{n}(Z|q)=\sum_{s\geq \max (n,m)}\frac{1}{[\frac{s-m}{2}]_{q}![%
\frac{s-n}{2}]_{q}!} \\
\times \sum_{k=\frac{s}{2}-\min (\frac{n}{2},\frac{m}{2})}^{\min (\frac{n}{2}%
,\frac{m}{2})+\frac{s}{2}}\rho _{23}^{k}(\rho _{12}\rho _{13})^{s-k}\QATOPD[ 
] {m}{k-\frac{s-m}{2}}_{q}\QATOPD[ ] {n}{k-\frac{s-n}{2}}%
_{q}[k]_{q}![s-k]_{q}!\text{.}
\end{gather*}

In particular 
\begin{equation}
\limfunc{cov}(Y,Z)=\frac{(\rho _{23}+\rho _{12}\rho _{13})}{(1-rq)}\text{.}
\label{cov}
\end{equation}
\end{lemma}

\begin{proof}
i) Let us calculate $\int_{S(q)}H_{2n}(z|q)f_{Z}(z|r,q)dz\text{.}$ Using (%
\ref{r2}) we have 
\begin{equation*}
\int_{S(q)}H_{2n}(z|q)f_{Z}(z|r,q)dz=(1-r)\sum_{k=0}^{\infty }\frac{r^{k}}{%
[k]_{q}!(r)_{k+1}}\int_{S(q)}H_{2n}(z|q)H_{2k}(z|q)f_{N}(z|q)dz\text{.}
\end{equation*}%
Hence 
\begin{equation*}
\int_{S(q)}H_{2n}(z|q)f_{Z}(z|r,q)dz=(1-r)\frac{r^{n}[2n]_{q}!}{%
[n]_{q}!(r)_{k+1}}.
\end{equation*}%
by (\ref{oH}). Recall also that $\frac{(1-r)}{(r)_{k+1}}\allowbreak
=\allowbreak \frac{1}{(rq)_{k}}$. For $n=1$ we get $EH_{2}(Z|q)=r\frac{1+q}{%
1-rq}\text{.}$ Now recall that $H_{2}(x|q)=x^{2}-1.$

ii) We have, denoting $r=$ $\rho _{12}\rho _{13}\rho _{23}:$ 
\begin{gather*}
EH_{m}(Y|q)H_{n}(Z|q)=(1-r)\sum_{j,k\geq 0}\frac{\rho _{23}^{k}(\rho
_{12}\rho _{13})^{j}}{(q)_{j}(q)_{k}} \\
\times
\int_{S(q)}H_{m}(y|q)H_{k}(y|q)H_{j}(y|q)f_{N}(y|q)dy%
\int_{S(q)}H_{n}(z|q)H_{k}(z|q)H_{j}(z|q)f_{N}(z|q)dz\text{.}
\end{gather*}%
Now we apply Lemma \ref{3her} and we have further: 
\begin{gather*}
EH_{m}(Y|q)H_{n}(Z|q)=(1-r) \\
\times \sum_{j,k\geq 0}\rho _{23}^{k}(\rho _{12}\rho _{13})^{j}\frac{%
[n]_{q}![m]_{q}![j]_{q}![k]_{q}!}{[\frac{n+j-k}{2}]_{q}![\frac{m+j-k}{2}%
]_{q}![\frac{n+k-j}{2}]_{q}![\frac{m+k-j}{2}]_{q}![\frac{j+k-m}{2}]_{q}![%
\frac{j+k-n}{2}]_{q}!} \\
=(1-r)\sum_{s\geq \max (n,m)}\frac{1}{[\frac{s-m}{2}]_{q}![\frac{s-n}{2}%
]_{q}!}\times \\
\sum_{k=\frac{s}{2}-\min (\frac{n}{2},\frac{m}{2})}^{\min (\frac{n}{2},\frac{%
m}{2})+\frac{s}{2}}\rho _{23}^{k}(\rho _{12}\rho _{13})^{s-k}\QATOPD[ ] {m}{%
k-\frac{s-m}{2}}_{q}\QATOPD[ ] {n}{k-\frac{s-n}{2}}_{q}[k]_{q}![s-k]_{q}!
\end{gather*}

For $m=n=1\text{ }$we see that $s+m$ as well as $s+n$ have to be even, hence 
$s$ has to be odd. Thus, since $H_{1}(x|q)=x$, we get 
\begin{gather*}
\limfunc{cov}(Y,Z)=EYZ=(1-r)\sum_{t=0}^{\infty }\frac{1}{[t]_{q}![t]_{q}!}%
\sum_{k=t}^{t+1}\rho _{23}^{k}(\rho _{12}\rho
_{13})^{2t+1-k}[k]_{q}![2t+1-k]_{q}!= \\
(1-r)\sum_{t=0}^{\infty }\frac{1}{[t]_{q}![t]_{q}!}(\rho _{23}^{t}(\rho
_{12}\rho _{13})^{t+1}[t]_{q}![t+1]_{q}!+\rho _{23}^{t+1}(\rho _{12}\rho
_{13})^{t}[t]_{q}![t+1]_{q}!)= \\
(1-r)(\rho _{23}+\rho _{12}\rho _{13})\sum_{t=0}^{\infty }r^{t}[t+1]_{q}=%
\frac{(1-r)}{1-q}(\rho _{23}+\rho _{12}\rho _{13})(\frac{1}{1-r}-\frac{q}{%
1-rq}) \\
=\frac{(\rho _{23}+\rho _{12}\rho _{13})}{1-rq}
\end{gather*}
\end{proof}

\subsection{Conditional distributions and conditional moments}

\subsubsection{\textbf{Conditional distributions}}

As far as conditional distributions are concerned, we have the following
simple result

\begin{proposition}
\label{cond}Following formulae (\ref{prod_c}) and (\ref{CN}) we have:%
\begin{gather}
X|(Y=y,Z=z)\sim f_{X|Y=y,Z=z}(x|y,z,\rho _{12},\rho _{13},\rho _{23},q) 
\notag \\
=f_{N}(x|q)\frac{\left( \rho _{12}^{2},\rho _{13}^{2}\right) _{\infty
}\prod_{i=0}^{\infty }\omega \left( y,z|\rho _{12}\rho _{13}q^{i}\right) }{%
(\rho _{12}^{2}\rho _{13}^{2})_{\infty }\prod_{i=0}^{\infty }(\omega
_{q}\left( x,y|\rho _{12}q^{i}\right) \omega _{q}\left( x,z|\rho
_{13}q^{i}\right) }  \notag \\
=\frac{f_{CN}(x|y,\rho _{12},q)f_{CN}(z|x,\rho _{13},q)}{f_{CN}(z|y,\rho
_{12}\rho _{13},q)},  \label{AW} \\
Y,Z|X=x\sim f_{Y,Z|X=x}(y,z|x,\rho _{12},\rho _{13},\rho _{23},q)  \notag \\
=\frac{f_{CN}(x|y,\rho _{12},q)f_{CN}(y|z,\rho _{23},q)f_{CN}(z|x,\rho
_{13},q)}{f_{CN}(x|x,r,q)}.  \label{RR}
\end{gather}
\end{proposition}

\subsubsection{\textbf{Conditional moments}}

It has to be noted in this sections that whenever one considers conditional
expectation, then all equalities are considered \textbf{almost surely} with
respect to the distributions of the conditioning random variable(s).

\begin{remark}
Following observation contained in the formula (\ref{AW}), the results
presented in \cite{SzablAW} and then in \cite{Szab-bAW} we deduce that
conditional distribution of $X|(Y=y,Z=z)$ is the distribution that makes
Askey-Wilson polynomials with complex but conjugate parameters orthogonal.
More precisely, we take: 
\begin{eqnarray}
a &=&\frac{\sqrt{1-q}}{2}\rho _{1}(y-i\sqrt{\frac{4}{1-q}-y^{2}}),
\label{par1} \\
b &=&\frac{\sqrt{1-q}}{2}\rho _{1}(y+i\sqrt{\frac{4}{1-q}-y^{2}}),
\label{par2} \\
c &=&\frac{\sqrt{1-q}}{2}\rho _{2}(z-i\sqrt{\frac{4}{1-q}-z^{2}}),
\label{par3} \\
d &=&\frac{\sqrt{1-q}}{2}\rho _{2}(z+i\sqrt{\frac{4}{1-q}-z^{2}}).
\label{par4}
\end{eqnarray}
\end{remark}

From \cite{SzablAW} (2.15) and \cite{Szab-bAW} (4.9) we deduce that the
conditional expectations \newline
$E(H_{n}(X|q)|Y,Z)$ are polynomials in $Z$ and $Y$ of order $n\text{.}$
Following these results we have two equivalent forms of this function: (\cite%
{SzablAW} (3.2) and \cite{SzablAW} Lemma 3ii)):

\begin{proposition}
\label{2na1}$E(H_{n}(X|q)|Y=y,Z=z)$ has one of the following equivalent form:

i) 
\begin{gather}
E(H_{n}(X|q)|Y=y,Z=z)  \label{1na2_1} \\
=\sum_{s=0}^{n}\QATOPD[ ] {n}{s}_{q}\rho _{12}^{n-s}\rho _{13}^{s}\left(
\rho _{12}^{2}\right) _{s}H_{n-s}\left( y|q\right) P_{s}\left( z|y,\rho
_{12}\rho _{13},q\right) /(\rho _{12}^{2}\rho _{13}^{2})_{s}\text{,}  \notag
\end{gather}
where $\left\{ P_{s}\left( z|y,\rho _{1}\rho _{2},q\right) \right\} _{s\geq
-1}$ constitute the so called Al-Salam--Chihara polynomials given by (\ref%
{asc}).

ii) 
\begin{gather}
E(H_{n}(X|q)|Y=y,Z=z)=\frac{1}{\left( \rho _{12}^{2}\rho _{13}^{2}\right)
_{n}}\sum_{k=0}^{\left\lfloor n/2\right\rfloor }(-1)^{k}q^{\binom{k}{2}}%
\QATOPD[ ] {n}{2k}_{q}\QATOPD[ ] {2k}{k}_{q}\left[ k\right] _{q}!\rho
_{13}^{2k}\rho _{12}^{2k}  \label{1na2_2} \\
\times \left( \rho _{12}^{2},\rho _{13}^{2}\right) _{k}\sum_{j=0}^{n-2k}%
\QATOPD[ ] {n-2k}{j}_{q}\left( \rho _{12}^{2}q^{k}\right) _{j}\left( \rho
_{13}^{2}q^{k}\right) _{n-2k-j}\rho _{12}^{n-2k-j}\rho _{13}^{j}H_{j}\left(
z|q\right) H_{n-2k-j}(y|q)\text{.}  \notag
\end{gather}

iii) $E(P_{n}(X|y,\rho _{12},q)|Y=y,Z=z)=\frac{\rho _{13}^{n}(\rho
_{12}^{2})_{n}}{(\rho _{12}^{2}\rho _{13}^{2})_{n}}P_{n}(z|y,\rho _{12}\rho
_{13},q).$

In particular we have 
\begin{equation}
E(X|Y=y,Z=z)=\frac{y\rho _{12}(1-\rho _{13}^{2})+z\rho _{13}(1-\rho
_{12}^{2})}{1-\rho _{12}^{2}\rho _{13}^{2}}.  \label{ex}
\end{equation}
\end{proposition}

\begin{proof}
Following Remark \ref{cond} we recall that following \cite{SzablAW}(2.15),
properties of\newline
$f_{X|Y,Z}(x|y,z,\rho _{12},\rho _{13},\rho _{23},q)$ are basically known.
In particular following (3.2-3.3) of \cite{SzablAW} we get ii). Further
following Lemma 3 ii) of \cite{SzablAW} we get i). Finally following Thm.
4.1iii) of \cite{Szab-bAW} we get iii). Now taking either i) or iii) and $%
n=1 $ and $H_{1}(x|q)=x\text{,}$ $P_{1}(x|y,\rho ,q)=x-\rho y$ we get \ref%
{ex}.
\end{proof}

We also have similar result concerning one dimensional conditional moments:

\begin{theorem}
\label{c2d}One dimensional conditional moments say $E(H_{n}(Y|q)|Z=z)$ are
polynomials of order not exceeding $n$ in $Z\text{.}$ More precisely for $%
n=2m+1$ we have 
\begin{equation*}
E(H_{2m+1}(Y|q)|Z=z)=\sum_{s=0}^{m}\QATOPD[ ] {2m+1}{s}_{q}(\rho
_{23}^{2m+1-s}+(\rho _{12}\rho _{13})^{2m+1-s})W_{s,2m+1-s}(z|r,q)\text{,}
\end{equation*}%
and for $n=2m\text{,}$ $m\geq 1$ we have: 
\begin{gather*}
E(H_{2m}(Y|q)|Z=z)=\QATOPD[ ] {2m}{m}_{q}r^{m}W_{m,m}(z|r,q) \\
+\sum_{s=0}^{m-1}\QATOPD[ ] {2m}{s}_{q}(\rho _{23}^{2m-s}+(\rho _{12}\rho
_{13})^{2m-s})W_{s,2m-s}(z|r,q)\text{,}
\end{gather*}%
where polynomials $\left\{ W_{m,k}(z|r,q)\right\} _{m,k\geq 0}$ are given by
(\ref{ww}).

In particular we have:

\begin{equation}
E(Y|Z=z)=\frac{(\rho _{23}+\rho _{12}\rho _{13})}{(1+r)}z\text{,}
\label{cyz}
\end{equation}%
and

\begin{equation}
E(Y^{2}|Z=z)=\frac{(\rho _{23}^{2}+\rho _{12}^{2}\rho
_{13}^{2})(1-qr)+r(1-r)(1+q)}{(1+r)(1-qr^{2})}z^{2}+\frac{1+r^{2}-\rho
_{23}^{2}-\rho _{12}^{2}\rho _{13}^{2}}{(1-qr^{2})}\text{.}  \label{cy2z}
\end{equation}
\end{theorem}

\begin{proof}
We start with formula (\ref{3pol}) that we apply to expression for $%
E(H_{n}(Y|q)|Z=z)$ that is equal to 
\begin{gather*}
E(H_{n}(Y|q)|Z=z)f_{R}(z|r,q)=(1-r)f_{N}(z|q)\sum_{j,i\geq 0}\frac{\rho
_{23}^{i}(\rho _{12}\rho _{13})^{j}}{[i]_{q}![j]_{q}!}H_{i}(z|q)H_{j}(z|q) \\
\times \int_{S(q)}H_{n}(y|q)H_{i}(y|q)H_{j}(y|q)f_{N}(y|q)dy= \\
(1-r)f_{N}(z|q)\sum_{j,i\geq 0}\frac{\rho _{23}^{i}(\rho _{12}\rho _{13})^{j}%
}{[i]_{q}![j]_{q}!}H_{i}(z|q)H_{j}(z|q)\frac{[j]_{q}![i]_{q}![n]_{q}!}{[%
\frac{i+j-n}{2}]_{q}![\frac{i+n-j}{2}]_{q}![\frac{j+n-i}{2}]_{q}!}= \\
(1-r)f_{N}(z|q)\sum_{j,i\geq 0}\frac{\rho _{23}^{i}(\rho _{12}\rho
_{13})^{j}[n]_{q}!}{[\frac{i+j-n}{2}]_{q}![\frac{i+n-j}{2}]_{q}![\frac{j+n-i%
}{2}]_{q}!}H_{i}(z|q)H_{j}(z|q)\text{.}
\end{gather*}%
Now if $n$ is odd, say of the form $n=2m+1\text{,}$ $m\geq 0$ we notice that
all three numbers $\frac{i+j-n}{2},\frac{i+n-j}{2},\frac{j+n-i}{2}$ are
simultaneously nonnegative and integer for $j=i+1,i+3,\ldots ,i+(2m+1)$ and $%
i=j+1,j+3,\ldots ,j+(2m+1)\text{.}$ Then these numbers are equal to
respectively $(i-m,m,m+1),\allowbreak (i-m+1,m-1,m+2),\allowbreak \ldots
,(i,0,2m+1)(j-m,m,m+1),\allowbreak (j-m+1,m-1,m+2),\allowbreak \ldots
,(j,0,2m+1)\text{.}$ If say $j=i+2(m-s)+1$ these numbers are equal to $%
(i-s,s,2m+1-s)$ we have the sum (keeping in mind that $r=\rho _{12}\rho
_{13}\rho _{23}):$ 
\begin{gather*}
(1-r)f_{N}(z|q)\sum_{i\geq s}\frac{\rho _{23}^{i-s}[2m+1]_{q}!(\rho
_{12}\rho _{13})^{i+2(m-s)+1}}{[i-s]_{q}![s]_{q}![2m+1-s]_{q}!}%
H_{i}(z|q)H_{i+2(m-s)+1}(z|q)= \\
(1-r)f_{N}(z|q)(\rho _{12}\rho _{13})^{2m-s+1}\QATOPD[ ] {2m+1}{s}%
_{q}\sum_{l\geq 0}\frac{r_{23}^{l}}{[l]_{q}!}H_{l+s}(z|q)H_{l+(2m-s+1)}(z|q)%
\text{.}
\end{gather*}%
Now recall (\ref{simp}) that states that $\sum_{l\geq 0}\frac{r_{23}^{l}}{%
[l]_{q}!}H_{l+s}(z|q)H_{l+(2m-s+1)}(z|q)$ is equal to some polynomial in $z$
of order $s+(2m-s)+1=2m+1$ times $\sum_{l\geq 0}\frac{r_{23}^{l}}{[l]_{q}!}%
H_{l}^{2}(z|q)$. But $(1-r)f_{N}(z|q)\sum_{l\geq 0}\frac{r_{23}^{l}}{[l]_{q}!%
}H_{l}^{2}(z|q)$ is the marginal density of $Z\text{.}$ Thus we see that for 
$s=0,\ldots ,m$ we get linear combination of some (depending on $s)$
polynomials of order $2m+1$ in $z$ with coefficients equal to $\QATOPD[ ] {%
2m-1}{s}_{q}(\rho _{12}\rho _{13})^{2m-s+1}$ times the density of $%
f_{Z}(z|r,q)$. If say $i=j+2(m-s)+1$ then we have similar situation. The
only difference lies in fact that this time linear combination is with $%
\QATOPD[ ] {2m+1}{s}_{q}(\rho _{23})^{2m-s+1}$.

Similar situation is when $n=2m$, $m\geq 1$. Then numbers $\frac{i+j-n}{2},%
\frac{i+n-j}{2},\frac{j+n-i}{2}$ are simultaneously nonnegative and integer
for $i=j,j+2,\ldots ,j+2m$ and $j=i+2,\ldots i+2m\text{.}$ Altogether we
will have $2m+1$ summands. Again each of them will be equal to some
polynomial of order $2m$ in $z$ times $\sum_{l\geq 0}\frac{r_{23}^{l}}{%
[l]_{q}!}H_{l}^{2}(z|q)$ with different (depending on if $j=i+2s$ or $%
i=j+2s) $ coefficients of linear combination.

Now let us apply these ideas to cases $k=1$ and $k=2.$ Recall that $%
H_{1}(x|q)=x$ and $H_{2}(x|q)=x^{2}-1\text{,}$ $W_{1,0}(x|r,q)=\frac{x}{1+r}%
\text{,}$ $W_{1,1}(x|r,q)=x^{2}\frac{1-r}{(1+r)(1-qr^{2})}+\frac{r}{%
(1-qr^{2})}\text{,}$ $W_{2,0}(x|r,q)=x^{2}\frac{1-qr}{(1+r)(1-qr^{2})}-\frac{%
1}{1-qr^{2}}$. Hence $E(Y|Z)$ is indeed given in ii). iii) we get likewise.
\end{proof}

\begin{corollary}
$\forall n,m\geq 0:E(H_{n}(X|q)H_{m}(Y|q)|Z)$ is a polynomial of order at
most $n+m$ of the conditioning random variable $Z.$

In particular we have:%
\begin{eqnarray}
E(XY|Z &=&z)=z^{2}\frac{\rho _{12}(\rho _{13}^{2}+\rho
_{23}^{2})(1-qr)+(1-r)(\rho _{13}\rho _{23}+qr\rho _{12})}{(1+r)(1-qr^{2})}
\label{cconv} \\
&&+\frac{\rho _{12}(1-\rho _{13}^{2})(1-\rho _{23}^{2})}{(1-qr^{2})}.  \notag
\end{eqnarray}
\end{corollary}

\begin{proof}
By the tower property of conditional expectation we have%
\begin{equation*}
E(H_{n}(X|q)H_{m}(Y|q)|Z=z)=E(H_{m}(Y|q)E(H_{n}(X|q)|(Y,Z))|Z=z)
\end{equation*}%
Using an assertion of Proposition \ref{2na1} we deduce that $%
E(H_{n}(X|q)|(Y=y,Z=z))$ is a polynomial of order $n$ in $y$ and $z.$ Hence $%
H_{m}(Y|q)E(H_{n}(X|q)|(Y=y,Z=z))$ is a polynomial of order $n+m$ in $y$ and 
$n$ in $z$ but together of order $n+m$ in both variables. It follows
formulae (\ref{1na2_1}) or (\ref{1na2_2}). Hence, indeed $%
E(H_{n}(X|q)H_{m}(Y|q)|=z)$ is a polynomial of order $n+m$ in $z.$

To get (\ref{cconv}) we combine (\ref{ex}), (\ref{cyz}) and (\ref{cy2z}) and
then use Mathematica.
\end{proof}

\section{Remarks and open problems\label{open}}

1. The most important remark seems to be an observation of the fact that all
conditional moments of order say $n$ are also polynomials of order $n$ in
the conditioning random variables. Among three-dimensional distributions
having similar property are the Gaussian distributions. Are there any others
except these two mentioned?

2. It could of interest to consider two extremal cases, i.e. $q=0$ and $%
q\rightarrow 1^{-}$.

2a. Following (\ref{Rp}), for $q=0$ we get: $w_{-1}(x|r,0)=0\text{,}$ $%
w_{0}(x|r,0)=1\text{,}$ $w_{1}(x|r,0)=x\text{,}$ $w_{2}(x|r,0)=x^{2}-1$ and
further for $n\geq 2$ the recursion (\ref{Rp}) becomes the following: 
\begin{equation*}
w_{n+1}(x|r,0)=xw_{n}(x|r,0)-w_{n-1}(x|r,0)
\end{equation*}%
which is the recurrence satisfied by the Chebyshev polynomials.
Consequently, we deduce that for $n\geq 1$ we have: 
\begin{equation*}
w_{n}(x|r,0)=U_{n}(x/2)-(2-r)U_{n-2}(x/2)\text{.}
\end{equation*}%
With little of algebra applied to (\ref{1wym}) with $q=0$ we get the
following density of one dimensional distribution: 
\begin{equation*}
f_{X}(x|r,0)=\frac{(1+r)\sqrt{4-x^{2}}}{2\pi ((1+r)^{2}-rx^{2})}\text{,}
\end{equation*}%
that is Kesten-McKay distribution.

2b. On the other hand again following (\ref{Rp}) for $q\rightarrow 1^{-}$ we
get: $w_{-1}(x|r,1)=0\text{,}$ $w_{0}(x|r,1)=1\text{,}$ and for $n\geq 0$ 
\begin{equation}
w_{n+1}(x|r,1)=xw_{n}(x|r,1)-\frac{1+r}{1-r}nw_{n-1}(x|r,1).  \label{wr}
\end{equation}%
Recall also that the so called probabilistic Hermite polynomials $\left\{
H_{n}(x)\right\} _{n\geq -1}$ i.e. orthogonal with respect to $\exp
(-x^{2}/2)/\sqrt{2}$ satisfy the following recurrence: 
\begin{equation}
H_{n+1}(x)=xH_{n}(x)-nH_{n-1}(x).  \label{he}
\end{equation}%
Hence combining (\ref{wr}) and (\ref{he}) we see that 
\begin{equation*}
w_{n}(x|r,1)=(\frac{1+r}{1-r})^{n/2}H_{n}(\sqrt{\frac{1-r}{1+r}}x)\text{,}
\end{equation*}%
for all $n\geq -1.$ Thus we deduce that the one dimensional marginal has
density equal to 
\begin{equation*}
f_{X}(x|r,1)=\sqrt{\frac{(1-r)}{2\pi (1+r)}}\exp (-\frac{(1-r)x^{2}}{2(1+r)})%
\text{.}
\end{equation*}

3. Let $(Y,Z)$ be two random variables. Suppose that we know the marginal
distribution of $Y$. The property that all conditional moments of order $n$
say of $Y$ are polynomials of the same order in conditioning random variable
(i.e. $Z$) for $n\allowbreak \geq 1$ was called general polynomial
regression property (GRP) of $Y$ given $Z$ in \cite{BrSz90}. Also, there the
property that additionally the coefficients by the greatest powers (say $n)$
are of the form $\rho ^{n}$, where $\rho $ denotes here correlation
coefficient between was called there simply polynomial regression (PR)
property of $Y$ given $Z.$ In \cite{BrSz90} it was shown that if both
marginals are standardized normal and both have GRP and one of then has PR
property, then necessarily joint distribution of $(Y,Z)$ must be binormal
distribution with parameters $(0,0;1,1,\rho ).$

Theorem \ref{c2d} states both, two random variables $(Y,Z)$ having joint
distribution with the density $f_{YZ}$ given by assertion ii) of Theorem \ref%
{margin} have GRP property. It seems that PR property is not satisfied by
none of the variables.

However the intriguing question seems to be what other property of
conditional moments should be additionally assumed in order to get the
characterization of joint distribution $f_{YZ}?$

\begin{acknowledgement}
The author is very grateful to two unknown referees. Especially to the one
of them, who has carefully read and checked all formulae and pointed out
typos and small mistakes.
\end{acknowledgement}


\begin{thebibliography}{99}
\bibitem{IA} Ismail, Mourad E. H. Classical and quantum orthogonal
polynomials in one variable. With two chapters by Walter Van Assche. With a
foreword by Richard A. Askey. Encyclopedia of Mathematics and its
Applications, 98. \emph{Cambridge University Press,} Cambridge, 2005.

\bibitem{Bryc2001S} Bryc, W\l odzimierz. Stationary random fields with
linear regressions. \emph{Ann. Probab.} \textbf{29} (2001), no. 1, 504--519.
MR1825162 (2002d:60014)

\bibitem{ISV87} Ismail, Mourad E. H.; Stanton, Dennis; Viennot, G\'{e}rard.
The combinatorics of \emph{q}-Hermite polynomials and the Askey-Wilson
integral. \emph{European J. Combin.} \textbf{8} (1987), no. 4, 379--392.
MR0930175 (89h:33015) xviii+706 pp. ISBN: 978-0-521-78201-2; 0-521-78201-5
MR2191786 (2007f:33001)

\bibitem{Szab5} Szab\l owski, Pawe\l\ J. Multidimensional \emph{q}-normal
and related distributions---Markov case. \emph{Electron. J. Probab.} \textbf{%
15} (2010), no. \textbf{40}, 1296--1318. MR2678392

\bibitem{Szab-bAW} Szab\l owski, Pawe\l\ J. Befriending Askey--Wilson
polynomials, \emph{Infin. Dimens. Anal. Quantum Probab. Relat. Top.} \emph{, 
}Vol . \textbf{17}, No. 3 (2014) 1450015 (25 pages),
http://arxiv.org/abs/1111.0601.

\bibitem{Szab-rev} Szab\l owski, Pawe\l\ J. On the \emph{q}-Hermite
polynomials and their relationship with some other families of orthogonal
polynomials. Demonstratio Math. 46 (2013), no. 4, 679--708. MR3136185,
http://arxiv.org/abs/1101.2875,

\bibitem{SzablAW} Szab\l owski, Pawe\l\ J. On the structure and
probabilistic interpretation of Askey-Wilson densities and polynomials with
complex parameters. \emph{J. Funct. Anal.} \textbf{261} (2011), no. 3,
635--659. MR2799574, http://arxiv.org/abs/1011.1541

\bibitem{Szablowski2009} Szab\l owski, Pawe\l\ J. \emph{q}-Gaussian
distributions: simplifications and simulations. \emph{J. Probab. Stat.}
2009, Art. ID 752430, 18 pp. MR2602881

\bibitem{bms} Bryc, W\l odzimierz; Matysiak, Wojciech; Szab\l owski, Pawe\l\ %
J. Probabilistic aspects of Al-Salam-Chihara polynomials. \emph{Proc. Amer.
Math. Soc.} \textbf{133} (2005), no. 4, 1127--1134 (electronic). MR2117214
(2005m:33033)

\bibitem{SzabP-M} Szab\l owski, Pawe\l\ J. Around Poisson-Mehler summation
formula. \emph{Hacet. J. Math. Stat.} \textbf{45} (2016), no. 6, 1729--1742.
MR3699734, http://arxiv.org/abs/1108.3024

\bibitem{BrSz90} Bryc, W\l odzimierz; Szab\l owski, Pawe\l\ J. Some
characteristic of normal distribution by conditional moments. Bull. Polish
Acad. Sci. Math. \textbf{38} (1990), no. 1-12, 209--218. MR1194266

\bibitem{Szab18} Szab\l owski, Pawe\l\ J., On generalized Kesten--McKay
distributions, submitted, http://arxiv.org/abs/1507.03191
\end{thebibliography}
\end{document}